\documentclass{amsart}[12pt]
\usepackage{graphicx,graphics, amsmath, amsthm, amscd, amsfonts}

\usepackage{latexsym}
\makeatletter \oddsidemargin.9375in \evensidemargin \oddsidemargin
\newtheorem{theorem}{Theorem}[section]
\newtheorem{lemma}[theorem]{Lemma}

\theoremstyle{definition}

\numberwithin{equation}{section}

\begin{document}
\Large
\title[Linear maps preserving the dimension of fixed points ... ]{Linear maps preserving the dimension of fixed points of operators }

\author[Ali Taghavi and Roja Hosseinzadeh]{Ali Taghavi and Roja
Hosseinzadeh}

\address{{ Department of Mathematics, Faculty of Mathematical Sciences,
 University of Mazandaran, P. O. Box 47416-1468, Babolsar, Iran.}}

\email{Taghavi@umz.ac.ir,  ro.hosseinzadeh@umz.ac.ir}



\begin{abstract}\large
Let $\mathcal{B(X)}$ be the algebra of all bounded linear operators on a complex Banach space $\mathcal{X}$ with $\dim \mathcal{X}\geq3$. In this paper, we characterize the forms of surjective linear maps on $\mathcal{B(X)}$ which preserve the dimension of the vector space containing of all fixed points of operators, whenever $\mathcal{X}$ is a finite dimensional Banach space. Moreover, we characterize the forms of linear maps on $\mathcal{B(X)}$ which preserve the vector space containing of all fixed points of operators.
\end{abstract} \maketitle

\section{Introduction And Statement of the Results}
\noindent
The study of maps on operator algebras
preserving certain properties is a topic which
attracts much attention of many authors (see \cite{1}--\cite{13} and the references cited there.) Some of these problems are concerned with preserving a certain property of usual products
or other products of operators (see \cite{1}--\cite{8} and \cite{12}--\cite{14}).

Let $B(X)$ denotes the algebra of all bounded linear operators on a complex Banach space $X$ with $\dim X\geq3$. Recall that $x \in X$ is a fixed point of an operator $T \in B(X)$, whenever we have $Tx=x$. Denote by $F(T)$,
the set of all fixed points of $T$. It is clear that for an linear operator $T$, $F(T)$ is a vector space. Denote by $ \dim F(T)$, the dimension of $F(T)$. In \cite{12}, authors characterized the forms of surjective maps on $\mathcal{B(X)}$ such that preserve the dimension of the vector space containing of all fixed points of products of operators. In this paper, we characterize the forms of surjective linear maps on $\mathcal{B(X)}$ which preserve the dimension of the vector space containing of all fixed points of operators, whenever $\mathcal{X}$ is a finite dimensional Banach space. Moreover, we characterize the forms of linear maps on $\mathcal{B(X)}$ which preserve the vector space containing of all fixed points of operators.
 The statements of our main results are the follows.
%---------------------------------------------------------------------------------------%
\begin{theorem} Let $\mathcal{X}$ be a complex Banach space with $\dim
\mathcal{X} \geq 3$. Suppose $\phi:\mathcal{B(X)}\longrightarrow \mathcal{B(X)}$ is a surjective linear map which satisfies the following condition:
$$F(A)= F( \phi(A)) \ \ \ \ \ (A\in \mathcal{B(X)}).$$
Then $ \phi (A)=A$ for all $A\in \mathcal{B(X)}$.
\end{theorem}
%---------------------------------------------------------------------------------------%
\begin{theorem} Let $n \geq3$. Suppose $\phi:\mathcal{M}_n\longrightarrow \mathcal{M}_n$ is a surjective linear map which satisfies the following condition:
$$ \dim F(A)= \dim F(\phi(A)) \ \ \ \ \ (A\in \mathcal{M}_n). \leqno{(2)}$$
Then there exists an invertible matrix $S\in \mathcal{M}_n$ such that $\phi(A)= SAS^{-1}$ or $\phi(A)= -SAS^{-1}$ for all $A\in \mathcal{M}_n$.
\end{theorem}
%---------------------------------------------------------------------------------------%
We recall some
notations. $X^ *$
denotes the dual space of $X$. For every nonzero $x\in \mathcal{X}$ and
$f\in \mathcal{X}^ *$, the symbol $x\otimes f$ stands for the rank one
linear operator on $\mathcal{X}$ defined by $(x\otimes f)y=f(y)x$ for any
$y\in \mathcal{X}$. Note that every rank one operator in $\mathcal{B(X)}$ can be
written in this way. The rank one operator $x\otimes f$ is
idempotent if and only if $f(x)=1$.

Let $x\otimes f$ be a rank-one operator. It is easy to check that $x\otimes f$ is an idempotent if and only if $F(x\otimes f)=\langle x\rangle$ (the linear subspace spanned by $x$). If $x\otimes f$ isn't idempotent, then $F(x\otimes f)=\{0\}$.

Given $P,Q\in \mathcal{P}$, we say that $P$ and $Q$ are orthogonal if and only if
$PQ=QP=0$.

%---------------------------------------------------------------------------------------%
\section{Linear maps preserving the fixed points of operators}
%---------------------------------------------------------------------------------------%
Assume that $\phi:\mathcal{B(X)}\longrightarrow \mathcal{B(X)}$ is a surjective linear map which satisfies the following condition:
$$F(A)= F( \phi(A)) \ \ \ \ \ (A\in \mathcal{B(X)}). \leqno{(1)}$$
 First we prove some elementary results
which are useful in the proof of Theorem 1.1.
%---------------------------------------------------------------------------------------%
\begin{lemma}
Let $x\in \mathcal{X}$ and $A\in \mathcal{B(X)}$. If $x$ and $Ax$ are linear independent vectors, then there exists a rank one idempotent $P$ such that $x \in F(A+P)$.
\end{lemma}
%---------------------------------------------------------------------------------------%
\begin{proof}
Since $x$ and $Ax$ are linear independent vectors, there exists a linear functional $f$ such that $f(x)=1$ and $f(Ax)=0$. Set $P= (x-Ax) \otimes f$. We have
$$ (A+P)x=(A+(x-Ax) \otimes f)x=Ax+(x-Ax)f(x)=x$$
which completes the proof.
\end{proof}
%---------------------------------------------------------------------------------------%
\begin{lemma}
$\phi(P)=P$ for every rank one idempotent $P$.
\end{lemma}
%---------------------------------------------------------------------------------------%
\begin{proof}
Since $\mathcal{X}= F(I)=F( \phi(I))$, we obtain $\phi(I)=I$. This implies that $ \ker (A)= \ker ( \phi (A))$, for every $A\in \mathcal{B(X)}$, because
$$x \in \ker (A)  \Leftrightarrow  Ax+x=x \Leftrightarrow x \in F(A+I)= F( \phi (A)+I) $$
$$ \Leftrightarrow  x= \phi (A) x+x \Leftrightarrow x \in \ker ( \phi (A)).$$
Let $P= x \otimes f$, for some $x\in \mathcal{X}$ and $f\in \mathcal{X}^*$ such that $f(x)=1$. We have
$$ \ker f(x)= \ker x \otimes f= \ker \phi ( x \otimes f)$$
which implies that $\phi ( x \otimes f)$ is a rank one operator, because its null space is a hyperspace of $\mathcal{X}$. Thus there exists a $y\in \mathcal{X}$ such that $ \phi (P)= y \otimes f$, because its null space is equal to the null space of $f$. By assumption we have $F( x \otimes f)= F( y \otimes f)$ which implies that $x$ and $y$ are linear dependent and also $f(y)=1$. Hence $x=y$ and this completes the proof.
\end{proof}
%---------------------------------------------------------------------------------------%
\begin{lemma}
For any rank one idempotent $P$, there exists an $ \eta (P) \in \mathbb{C}$ such that $\phi(A)+P= \eta (P)(A+P)$, for every $A \in \mathcal{B(X)}$.
\end{lemma}
%---------------------------------------------------------------------------------------%
\begin{proof}
Let $P$ be a rank one idempotent. Suppose there exists an $x \in \mathcal{X}$ such that $( \phi(A)+P)x$ and $(A+P)x$ are linear independent. If $x$ and $(A+P)x$ are linear independent, by Lemma 2.1 there exists a rank one idempotent $Q$ such that $x \in F(A+P+Q)$ which by Lemma 2.2 implies that $x \in F( \phi (A)+P+Q)$. Thus $\phi (A)x=Ax$, which is a contradiction, because we assume that $( \phi(A)+P)x$ and $(A+P)x$ are linear independent. Therefore, $ \phi(A)+P$ and $A+P$ are locally linear dependent for any rank one idempotent $P$. By [9, Theorem
2.4], there is an $ \eta (P) \in \mathbb{C}$ such that $\phi(A)+P= \eta (P)(A+P)$.
\end{proof}
%---------------------------------------------------------------------------------------%
\par \vspace{.3cm}
\noindent \textbf{Proof of Theorem 1.1}. Let $A \in \mathcal{B(X)}$ and $P$ be a rank one idempotent. It is enough to prove that $ \eta (P)$ in Step 2 is equal to $1$. If $A+P$ is a scalar operator, then by Lemma 2.3, $ \eta (P)=1$, because $ A+P= \phi (A+P)= \phi (A)+P$. If $A+P$ is a non-scalar operator, then there exists an $x \in \mathcal{X}$ such that $x$ and $(A+P)x$ are linear independent. By Lemma 2.1 we can find a rank one idempotent $Q_1$ such that $x \in F(A+P+Q_1)$ which implies that
$$x \in F( \phi (A)+P+Q_1)= F( \eta (P)(A+P)+Q_1)$$
and so
$$ (A+P)x= \eta (P)(A+P)x.$$
Since $ (A+P)x$ is nonzero, $\eta (P)=1$ and this completes the proof.

%---------------------------------------------------------------------------------------%
\section{Linear maps preserving the dimension of fixed points of operators}
%---------------------------------------------------------------------------------------%
Let $n \geq3$. Assume that $\phi:\mathcal{M}_n\longrightarrow \mathcal{M}_n$ is a surjective linear map which satisfies the following condition:
$$ \dim F(A)= \dim F(\phi(A)) \ \ \ \ \ (A\in \mathcal{M}_n). \leqno{(2)}$$
First we prove some elementary results
which are useful in the proof of Theorem 1.12.
%---------------------------------------------------------------------------------------%
\begin{lemma}
$\phi (A)=I$ if and only if $A=I$.
\end{lemma}
%---------------------------------------------------------------------------------------%
\begin{proof}
Since $n= \dim F(I)= \dim F( \phi(A))= \dim F( A)$, we obtain $F(A)= \mathcal{X}$ and so $A=I$.
\end{proof}
%---------------------------------------------------------------------------------------%
\begin{lemma}
$\phi$ preserves the rank one idempotents in both directions.
\end{lemma}
%---------------------------------------------------------------------------------------%
\begin{proof}
It is easy to check that
 $$ \ker (A)= F ( A+I), \leqno{(*)}$$
 for every $A\in \mathcal{B(X)}$. Let $x\in \mathcal{X}$ and $f\in \mathcal{X}^*$. Lemma 3.1 together with $(*)$ and $(2)$ implies that
 \begin{eqnarray*}
n-1&=&\dim \ker(x \otimes f)\\
&=& \dim F(x \otimes f+I)\\
&=& \dim F( \phi (x \otimes f)+I)\\
&=&\dim \ker( \phi (x \otimes f)).
\end{eqnarray*}
Thus $\phi (x \otimes f)$ is a rank one operator. If $f(x)=1$, then we have
$$1= \dim F(x \otimes f)= \dim F( \phi(x \otimes f))$$
which implies that $\phi (x \otimes f)$ is idempotent. In a similar way can show that $\phi$ preserves the rank one idempotents in other direction and this completes the proof.
\end{proof}
%---------------------------------------------------------------------------------------%
\begin{lemma}
$\phi$ preserves the orthogonality of rank one idempotents in both directions.
\end{lemma}
%---------------------------------------------------------------------------------------%
\begin{proof}
It is easy to check that
$$ \mathrm{rank}(A) \geq \dim F(A) \leqno{(**)}$$
for every $A \in \mathcal{M}_n $. Let $P$ and $Q$ be two orthogonal idempotents. This together with $(2)$ implies that
 $$2= \dim F(P+Q)=  \dim F( \phi (P) + \phi (Q)).$$
 So by $(**)$ we obtain $ \mathrm{rank}( \phi (P) + \phi (Q)) \geq 2$.
 On the other hand, since $\phi (P)$ and $\phi (Q)$ are rank one idempotents, $ \mathrm{rank}( \phi (P) + \phi (Q)) \leq 2$. Thus we obtain $ \mathrm{rank}( \phi (P) + \phi (Q))=2$ and this yields the orthogonality of $\phi (P)$ and $\phi (Q)$. In a similar way can show that $\phi$ preserves the orthogonality of rank one idempotents in other direction and this completes the proof.
\end{proof}
%---------------------------------------------------------------------------------------%
\par \vspace{.3cm}
\noindent \textbf{Proof of Theorem 1.2}. By Lemmas 3.2 and 3.3, it is clear that $\phi$ preserves idempotent operators in both directions. So assertion follows from [11, Theorem 2.2].
%---------------------------------------------------------------------------------------%
\par \vspace{.3cm} In the end of this paper, we pose the following problem.
%---------------------------------------------------------------------------------------%
\par \vspace{.3cm}
\noindent \textbf{Problem}. In this paper, we discuss the linear maps on matrix algebra preserving the dimension of fixed points of elements. One natural problem is how one should characterize linear maps on $\mathcal{B(X)}$ preserving the dimension of fixed points of elements, where $\mathcal{X}$ is infinite dimensional Banach space.
%---------------------------------------------------------------------------------------%
\par \vspace{.4cm}{\bf Acknowledgements:} This research is partially
supported by the Research Center in Algebraic Hyperstructures and
Fuzzy Mathematics, University of Mazandaran, Babolsar, Iran.
%---------------------------------------------------------------------------------------%
\bibliographystyle{amsplain}

\end{document}